\documentclass[11pt,reqno]{amsart}
\usepackage{amssymb,mathrsfs,graphicx,enumerate}
\usepackage{amsmath,amsfonts,amssymb,amscd,amsthm,bbm}
\usepackage[retainorgcmds]{IEEEtrantools}
\usepackage{colortbl}
\usepackage{lipsum}
\usepackage{graphicx, subfigure}
\usepackage[caption=false]{subfig}
\usepackage{graphicx}
\title[Higher-order interaction model] 
{Higher-order interaction model from geometric measurements}

\author[D. Kim]{Dohyun Kim}
\address[Dohyun Kim]{\newline School of Mathematics, Statistics and Data Science, \newline Sungshin Women's University, Seoul 02844, Republic of Korea}
\email{dohyunkim@sungshin.ac.kr}

\author[H. Park]{Hansol Park}
\address[Hansol Park]{\newline Department of Mathematics,\newline Simon Fraser University, 8888 University Dr., Burnaby, BC V5A 1S6, Canada}
\email{hansol\_park@sfu.ca}
\email{hansol960612@snu.ac.kr}

\author[W. Shim]{Woojoo Shim}
\address[Woojoo Shim]{\newline School of Mathematics, Korea Institute for Advanced Study,\newline Seoul, 02455, Republic of Korea}
\email{cosmo.shim@gmail.com}

\newtheorem{theorem}{Theorem}[section]
\newtheorem{lemma}{Lemma}[section]
\newtheorem{corollary}{Corollary}[section]
\newtheorem{proposition}{Proposition}[section]
\newtheorem{remark}{Remark}[section]
\newtheorem{example}{Example}[section]

\newcommand{\bbr}{\mathbb R}

\newcommand{\dd}{\mathrm{d}}
\newcommand{\dt}{\mathrm{d}t}
\newcommand{\kp}{\kappa}

\makeatletter
\@namedef{subjclassname@2020}{\textup{2020} Mathematics Subject Classification}
\makeatother
\begin{document}

\date{\today}


\thanks{\textbf{Acknowledgment.} The work of D. Kim was supported by the National Research Foundation of Korea (NRF) grant funded by the Korea government (MSIT) (No.2021R1F1A1055929), the work of H. Park was supported by Pacific Institute for the Mathematical Science(PIMS), Canada postdoctoral fellowship, and the work of W. Shim is partially supported by Samsung Science and Technology Foundation (SSTF-BA1401- 51).\\
\textbf{Corresponding Author:} Hansol Park}

\begin{abstract}

We introduce a higher simplicial generalization of the linear consensus model which shares several common features. The well-known linear consensus model is a gradient flow with a sum of squares of distances between each pair of points. Our newly suggested model is also represented as a gradient flow equipped with total $n$-dimensional volume functional consisting of $n+1$ points as a potential. In this manner, the linear consensus model coincides with the case of $n=1$ where distance is understood as the 1-dimensional volume. From a simple mathematical analysis, one can easily show that the linear consensus model (a gradient flow with 1-dimensional volume functional) collapses to one single point,  which can be considered as a 0-complex. By extending this result, we show that a solution to our model converges to an $(n-1)$-dimensional affine subspace. We also perform several numerical simulations with an efficient algorithm that reduces the computational cost. 

\end{abstract}

\maketitle \centerline{\date}


\section{Introduction} \label{sec:1}
Mathematical modeling allows us to analyze various natural phenomena and find optimal solutions. Two pioneers in this area, Kuramoto \cite{Ku, Ku2} and Winfree \cite{Wi}, catalyzed studying mathematical modeling and its applications. Especially, two mathematical models introduced by these two pioneers are pairwise interaction models which can be written in the following form:
\begin{align}\label{A-0}
\displaystyle\dot{x}_i=\sum_{j=1}^N I_{ij}(x_i, x_j),\quad \forall~i\in [N]:=\{1, 2, \cdots, N\},
\end{align}
where $N$ is the number of the particles and $I_{ij}$ is the interaction depending on $x_i$ and $x_j$, i.e., the dynamics of $i$-th particle is determined by composition of the interactions between $i$-th and $j$-th particles for all $1\leq j\leq N$. Inheriting this common form from two models introduced in \cite{Ku2, Wi}, a lot of researchers have focused on these pairwise interaction models \cite{BGL, CLC, DB, JC, KJ, Lo2,  TB}. Also, these models are applied in various areas: swarming of robots \cite{MPG}, pattern formation of biological groups \cite{TB}, unmanned aerial vehicle \cite{CCF}.
However, interactions in real-world systems are more complicated than these pairwise interaction models. Lately, to consider more complicated interactions, constructing higher-order interaction models is in the limelight. Since there are various ways to generalize the pairwise interaction model to a higher-order interaction model, lots of models are derived independently: a higher-order model has been assumed directly obtained from replacing the interaction term in \eqref{A-0} in \cite{GB}, a higher-order model has been obtained from generalizing the potential function analytically in \cite{Lo1, Lo3}. Ecologists also use these higher-order models in \cite{SB, YMB} practically. 

Although the modeling starts from different ideas, the common form of higher-order models can be written as 
\begin{align}\label{A-01}
\dot{x}_i=\sum_{j_1,\cdots, j_m=1}^NI_{ij_1\cdots j_m}(x_i, x_{j_1}, \cdots, x_{j_m}),\quad\forall~i\in[N],
\end{align}
where $I_{ij_1\cdots j_m}$ is the interaction depends on $(m+1)$ particles $x_i$, $x_{j_1}$, $\cdots$, and $x_{j_m}$. As we can compare two systems \eqref{A-0} and \eqref{A-01}, the interaction term depends on more than two particles. The goal of this paper is to introduce a new interaction term $I_{ij_1\cdots j_m}$ which might be considered somewhat natural according to its geometrical interpretation.

As a smooth warmup, we begin with the linear consensus model as the simplest aggregation model on $\bbr^d$:
\begin{align}\label{B-1}
\begin{cases}
\displaystyle\dot{x}_i=\frac{\kappa_1}{N}\sum_{k=1}^N (x_k-x_i),\quad t>0,\\
x_i(0)=x_i^0\in\bbr^d,\quad\forall~i\in[N],
\end{cases}
\end{align}
where $\kappa_1$ is a non-negative coupling strength and $N$ is the number of particles. Recall that the linear consensus model can be considered as the Kuramoto model without the constraint of particles on the circle(refer \cite{Lo1}).

One of the notable features of \eqref{B-1} is that it can be represented as a gradient flow with an analytic potential. More precisely, if we define 
\begin{align}\label{B-2}
\mathcal V_1(\mathcal X) :=\frac{\kappa_1}{4N}\sum_{k, \ell=1}^N\|x_k-x_\ell\|^2,\quad \mathcal X :=(x_1,\cdots,x_N)\in \bbr^{dN},
\end{align}
then \eqref{B-1} is written as 
\begin{align*}
\dot{x}_i=-\nabla_{x_i} \mathcal V_1(\mathcal X).
\end{align*}
The use of subscript $1$ in $\mathcal{V}_1$ and $\kappa_1$ is to use a coherent notation with the higher-order potentials in later argument. Furthermore, simple calculation straightforwardly yields the conservation of the center of mass defined by $\bar{x}:=\frac{1}{N}\sum_{k=1}^Nx_k$, and the explicit solution to (1.3) can be obtained as follows:
\[
x_i(t)=(1-e^{-\kappa_1t})\bar{x}^0+e^{-\kappa_1t}x_i^0,\quad\forall~i\in[N].
\]
Hence, all $x_i$ converges exponentially to the initial center of mass and potential $\mathcal V_1(\mathcal X)$  also vanishes exponentially. In this regard, we would say that system \eqref{B-1} is a gradient flow equipped with total distance functional as a potential, and for the asymptotic behavior, a solution to the system collapses to a common point and the potential converges to zero. 

Our work is dedicated to generalizing \eqref{B-1} to a new system while preserving several similar properties of \eqref{B-1}. We extend the concept of distance (or length) between to a one-dimensional volumne of 1-simplex.  In the same vein, the area bounded by three points can be understood as a two-dimensional volume of 2-simplex. In this manner, we are interested in the following simple question:\\
\begin{quote}
(Q): ``What if we consider a gradient flow with the potential as a total squared $n$-dimensional volume functional instead of length functional?''\\
\end{quote}

The main results of this paper are three-fold. First, we introduce a new gradient flow where the distance between two points in $\mathcal{V}_1$ is generalized to $n$-dimensional volume functional between $(n+1)$ points. More precisely, for $n+1$ points among $N$ points($n\ll N$) in $\bbr^d$, we consider the $n$-dimensional volume of those points. e.g., length for $n=2$, volume for $n=3$, etc. Second, we study the asymptotic behavior of the proposed system. For the linear consensus model (the case of $n=1$), it is well-known that a solution always converge to a single point. In other words, a solution to the gradient flow with 1-dimensional volume tends to a 0-simplex. Hence, since our system is a gradient flow with $n$-dimensional volume, it is natural to expect that a solution converges to an $(n-1)$-dimensional affine subspace. We indeed show that this is true. Lastly, we suggest a reduced model which exhibits similar emergent behaviors to reduce the computational cost. In order to obtain the desired convergence toward affine subspace, it suffices to consider a fewer interaction.

The rest of this paper is organized as follows. We generalize the linear consensus model's potential and construct a new model in Section \ref{sec:3}. Using some geometric property to represent the equilibrium set, we provide the long time behaviors of the system in Section \ref{sec:4}. Since the computational cost is too large when we consider the higher dimensional simplex, we discuss the way to reduce the computational while preserving a similar long time behaviors in Section \ref{sec:5}. Finally, Section \ref{sec:6} devotes to the conclusion of the paper.
\\
\newline
\textbf{Notations.} 
For $n\geq d$, we define the set of $n$-dimensional affine subspace in $\bbr^d$ as follows:
\[
\mathcal{A}_d^n:=\left\{P\subset \bbr^d: P\text{ is $n$-dimensional affine subspace}\right\}.
\]


\section{Construction of models}\label{sec:3}
\setcounter{equation}{0}
In this section, we introduce a new multi-particle interaction model naturally generalized from linear consensus model \eqref{B-1}. Recall that the linear consensus model is formulated as a gradient flow with potential $\mathcal V_1$ \eqref{B-2}, where the potential $\mathcal V_1$ is given as the sum of squared lengths between every two points. Since \textit{length} is a one-dimensional object, we generalize it to higher dimensional object. To this end, for any fixed natural number $n \geq 1$, let $\mathrm{Vol}_n(x_1, x_2, \cdots, x_{n+1})$ be the $n$-dimensional volume of $n$-simplex consisting of $n+1$ points $\{ x_1, x_2, \cdots, x_{n+1}\}$. Since $\|x_i-x_j\|$ is the length between two points $x_i$ and $x_j$ , it can be understood as the $1$-dimensional volume of $1$-simplex with vertices $x_i$ and $x_j$. If we slightly abuse the notation, we rewrite potential $\mathcal V_1$ as 
\begin{align}\label{C-0}
\mathcal V_1(\mathcal X)=\frac{\kappa_1}{4N}\sum_{j_1, j_2=1}^N\mathrm{Vol}_1(x_{j_1},x_{j_2})^2.
\end{align}
As the generalization of potential in \eqref{C-0}, we naturally define  
\begin{align}\label{C-0-1}
\mathcal V_n(\mathcal X) :=\frac{\kappa_n}{2(n+1)N^{n}}\sum_{j_1, \cdots, j_{n+1}=1}^N \mathrm{Vol}_n(x_{j_1}, \cdots, x_{j_{n+1}})^2,
\end{align}
where $\kp_n$ denotes the (attractive) coupling strength. Here, $\mathcal{V}_n$ is called \textit{$n$-simplex potential} with the coupling strength $\kp_n$. Then, the model reads as  
\begin{align}\label{C-1}
\begin{cases}
\displaystyle\dot{x}_i=-\nabla_{x_i} \mathcal V_n(\mathcal X)=-\frac{\kappa_n}{2N^n}\sum_{j_1, \cdots, j_{n}=1}^N\nabla_{x_i} \mathrm{Vol}_n(x_{j_1}, \cdots, x_{j_n}, x_i)^2,\quad t>0,\\
x_i(0)=x_i^0\in\bbr^d,\quad\forall~i\in[N].
\end{cases}
\end{align}
For newly proposed model \eqref{C-1}, explicit formula for $\mathrm{Vol}_n(x_{j_1}, \cdots, x_{j_{n+1}})$ is crucially required, and thanks to generalization of classical Heron's formula and the Cayley-Menger determinant \cite{So}, we can find the desired explicit formula. 
\begin{proposition}\label{P2.1}
For given $n+1$ points $\{x_1,\cdots,x_{n+1}\}$ in $\bbr^d$, let $B=(B_{ij})$ be an $(n+1)\times (n+1)$ matrix where each element $B_{ij}$ is given as $B_{ij}:=\|x_i-x_j\|^2$. Then, the $n$-dimensional volume for the $n$-simplex of $\{x_1,\cdots,x_{n+1}\}$ is given as 
\[
\mathrm{Vol}_n(x_1, \cdots, x_{n+1})^2=\frac{(-1)^{n+1}}{2^n(n!)^2}\mathrm{det}(\hat{B}),
\]
where  $\hat{B}$ is a matrix of size $(n+2)\times (n+2)$ obtained from $B$ by bordering $B$ with a top row $(0, 1, \cdots, 1)$ and a left column $(0, 1, \cdots, 1)^\top$. In other words,
\[
\hat{B} :=\begin{bmatrix}
0&\mathbf{1}_{n+1}^\top\\
\mathbf{1}_{n+1}&B
\end{bmatrix}, \quad \mathbf{1}_{n+1}:=(\underbrace{1, 1, \cdots, 1}_{(n+1)-\text{times}})^\top.
\]
\end{proposition}

Below, we introduce several examples for $\textup{Vol}_n$ with small numbers $n$. 
\begin{example}\label{E2.1}
(1) For $n=2$, $\textup{Vol}_2(x_1,x_2,x_3)$ is merely an area of the triangle whose vertex set is  $\{x_1,x_2,x_3\}$. By using the simplified notation $d_{ij}:= \|x_i - x_j\|$, Proposition \ref{P2.1} implies
\begin{align*}
\mathrm{Vol}_2(x_1, x_2, x_3)^2 &=-\frac{1}{16}\mathrm{det}\begin{pmatrix}
0&1&1&1\\
1&0&d_{12}^2&d_{13}^2\\
1&d_{12}^2&0&d_{23}^2\\
1&d_{13}^2&d_{23}^2&0
\end{pmatrix} \\
& =-\frac{1}{16}(-2d_{12}^2d_{23}^2-2d_{23}^2d_{31}^2-2d_{31}^2d_{12}^2+d_{12}^4+d_{23}^4+d_{31}^4),
\end{align*}
which coincides with the classical Heron's formula. \newline

\noindent(2) For $n=3$,the volume $\mathrm{Vol}_3(x_1, x_2, x_3, x_4)$ is the volume of tetrahedron whose vertex set is $\{x_1,x_2,x_3, x_4\}$, and Proposition \ref{P2.1} gives the following formula to $\mathrm{Vol}_3$:
\[
\mathrm{Vol}_3(x_1, x_2, x_3, x_4)^2=\frac{1}{288}\mathrm{det}\begin{pmatrix}
0&1&1&1&1\\
1&0&d_{12}^2&d_{13}^2&d_{14}^2\\
1&d_{12}^2&0&d_{23}^2&d_{24}^2\\
1&d_{13}^2&d_{23}^2&0&d_{34}^2\\
1&d_{14}^2&d_{24}^2&d_{34}^2&0
\end{pmatrix}.
\]
\end{example}
%
%

\section{Emergent behaviors}\label{sec:4}
\setcounter{equation}{0}
In this section, we study the emergent behaviors of systems \eqref{C-1}. As in the linear consensus model \eqref{B-1}, our model \eqref{C-1} also conserves the center of mass. 

\begin{proposition}[Conservation of center of mass]\label{P4.1}
Let $\mathcal{X}=\{x_i\}_{i=1}^N$ be a solution to system \eqref{C-1}  for some $n\geq1$. Then, the center of mass is conserved, i.e.,
\[
\frac1N \sum_{i=1}^Nx_i(t)=\frac1N \sum_{i=1}^Nx_i^0,\quad \text{or equivalently},\quad \bar{x}(t)=\bar{x}^0,\quad\forall~t\geq0.
\]
\end{proposition}
\begin{proof}
Fix an arbitrary position vector $\xi \in \mathbb{R}^d$, and consider a parameterized curve $s\mapsto \mathcal X^s:=(x_1+s\xi,\cdots,x_N+s\xi)$ where $s\in \bbr$. Then, one can easily verify that the $n$-simplex potential $\mathcal V_n$ satisfies
\begin{equation}\label{4.8}
\mathcal V_n(\mathcal X^s)\equiv V_n(\mathcal X^0),\quad \forall~s\in \mathbb{R}.
\end{equation}
By differentiating \eqref{4.8} with respect to $s$, the chain rule yields $\sum_{i=1}^{N}\nabla_{x_i}\mathcal V_n(\mathcal X^0)\cdot \xi=0$. 	Since $\xi$ can be any vector in $\mathbb{R}^d$, one has 
\[\sum_{i=1}^{N}\nabla_{x_i}\mathcal V_n(\mathcal X^0)\equiv 0,\quad \forall~x_1,\cdots,x_N\in \mathbb{R}^d, \]
which gives the desired result.
\end{proof}

Next, we show that every relative distance between two points is non-increasing along flow \eqref{C-1}. 

\begin{lemma}\label{L4.1}
	Let $\{x_i\}_{i=1}^N$ be a solution to system \eqref{C-1} for some $n\geq1$. For every $i, j\in[N]$, we have 
	\[
	\frac{\dd}{\dt}\|x_i-x_j\|^2\leq0,\quad t>0.
	\]
\end{lemma}

\begin{proof}
Since the proof of this lemma is lengthy, we provide the proof in Appendix \ref{app:A}.
\end{proof}

We combine Proposition \ref{P4.1} and Lemma \ref{L4.1} to show that the distance toward center of mass from the each particle is non-increasing along system \eqref{C-1}.

\begin{lemma}\label{L4.2}
Let $\mathcal{X}=\{x_i\}_{i=1}^N$ be a solution to system \eqref{C-1} for some $n\geq1$ with the initial configuration $\{x_i^0\}_{i=1}^N$. Then, for every $i\in [N]$, we have  
\[
\frac{\dd}{\dt}\|x_i-\bar{x}^0\|^2\leq 0, \quad\text{where}\quad \bar x^0 = \frac1N \sum_{j=1}^N x_j^0.
\]
\end{lemma}

\begin{proof}
The proof of this lemma is introduced in Appendix \ref{app:B}.
\end{proof}

It follows from Lemma \ref{L4.2} that $x_i(t)$ is uniformly bounded in time. More precisely, $x_i$ belongs to the ball centered at $\bar x^0$ with radius $R:=  \max_{k\in [N]}\|x_k^0-\bar{x}^0\|$, i.e., 
\[
 x_i(t) \subseteq K:=B\left(\bar{x}^0, R\right),\quad i\in [N],\quad t>0.
\]
This implies $\{x_i(t)\}_{i=1}^N$ does not escape the compact set $K$. Since system \eqref{C-1} is a gradient  system on a compact set,  there exists $x_i^\infty$ for each $i\in[N]$ such that
\[
\lim_{t\to\infty} x_i(t)=x_i^\infty.
\]

Using the convergence of $\{x_i(t)\}_{i=1}^N$ to an equilibrium of system \eqref{C-1}, we have the following theorem.

\begin{theorem}[Emergent behavior]\label{T4.1}
Let $\{x_i\}_{i=1}^N$ be a solution to system \eqref{C-1} for some $n\geq1$. Then, there exists $P^\infty\in \mathcal{A}_d^{n-1}$ such that 
\[\lim_{t\to\infty } x_i(t)=:x_i^\infty\in P^\infty,\quad \forall~i\in [N]. \]
In other words, there exists an $(n-1)$-dimensional affine subspace $P^\infty$ and all particles converge to points on $P^\infty$.
\end{theorem}

\begin{proof}
The proof of this theorem is introduced in Appendix \ref{app:C}.
\end{proof}

\begin{remark}\label{R4.2}
(1) If the set of initial data $\mathcal{X}^0:=\{x_i^0\}_{i=1}^N$ lies on an $(n-1)$-dimensional affine subspace $P$, then $\mathcal{X}^0$  is an equilibrium solution of system \eqref{C-1}. Together with Theorem \ref{T4.1}, we verify that  the following set becomes an equilibrium for \eqref{C-1}
\[
\mathcal{E}_n:=\left\{\{x_i\}_{i=1}^N: \{x_i\}_{i=1}^N\subset P,\quad P\text{ is an $(n-1)$-dimensional affine subspace}\right\}.
\]
Furthermore, we observe
\[
V_n(\mathcal{X})=0\quad \Longleftrightarrow\quad \mathcal{X}\in\mathcal{E}_n.
\]
Thus, we classify all equilibria for \eqref{C-1} and this implies that  if $\mathcal{X}=\{x_i\}_{i=1}^N$ is a solution to system \eqref{C-1}, then $V_n(\mathcal{X}(t))$ is a decreasing function and converges to zero. \\

\noindent(2) These $\mathcal{E}_n$ are  defined only for $1\leq n\leq d+1$ with the following hierarchy:
\[
\mathcal{E}_1\subsetneq \mathcal{E}_2\subsetneq\cdots \subsetneq \mathcal{E}_{d+1}=\bbr^d.
\]
\end{remark}

Now, we perform numerical simulation for \eqref{C-1} to support and visualize our theoretical results. For numerical implementation, we employ the fourth-order Runge--Kutta method and use the following system parameters:
\[
\Delta t = 10^{-3}, \quad 2\leq n\leq 3, \quad N=40.
\]
Here,  $\circ$ and $*$ marks represent the initial state and the final state, respectively.

\begin{figure}[h] 
\includegraphics[width=7cm]{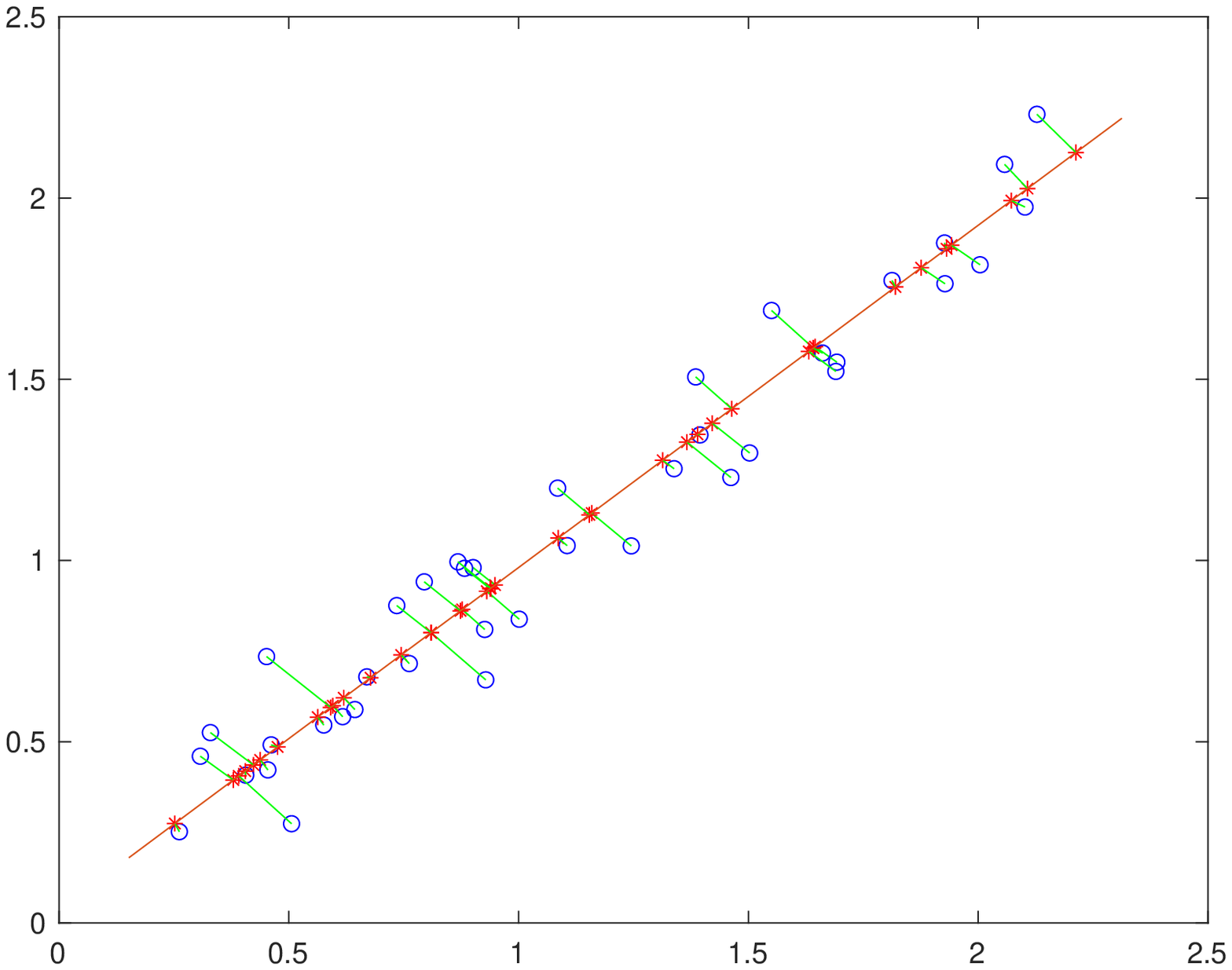}\quad
\includegraphics[width=7cm]{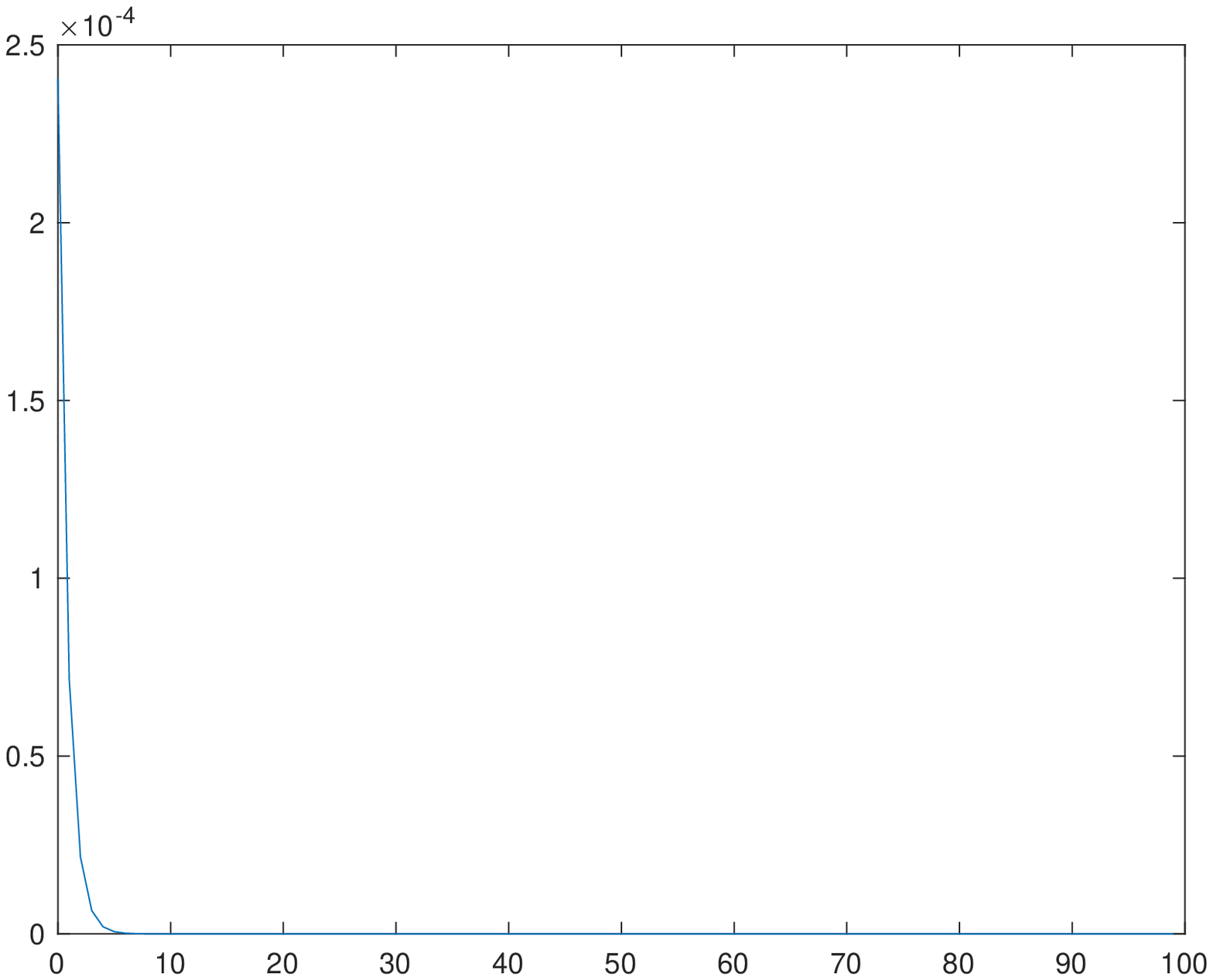}
\caption{If particles $\{x_i\}_{i=1}^N$ follows system \eqref{C-1} with $n=2$, then the particles are aligned on the same line(Left figure). The right graph plots the time evolution of the average area of triangles $x_ix_jx_k$. In this simulation, we choose the initial data as the perturbed points from the line. }\label{Fig1}
\end{figure}

In Figure \ref{Fig1}, we consider the case of $(d,n)=2$ and initial data randomly chosen from (small) perturbation of a given line colored in red. The left figure shows that all particles tend to align with the same line. For the right figure, we plot the temporal evolution of the averaged area of triangles with vertices $x_i,x_j,x_k$. As expected from theoretical results, we show that area of triangles converge to zero.

\begin{figure}[h] 
\includegraphics[width=7cm]{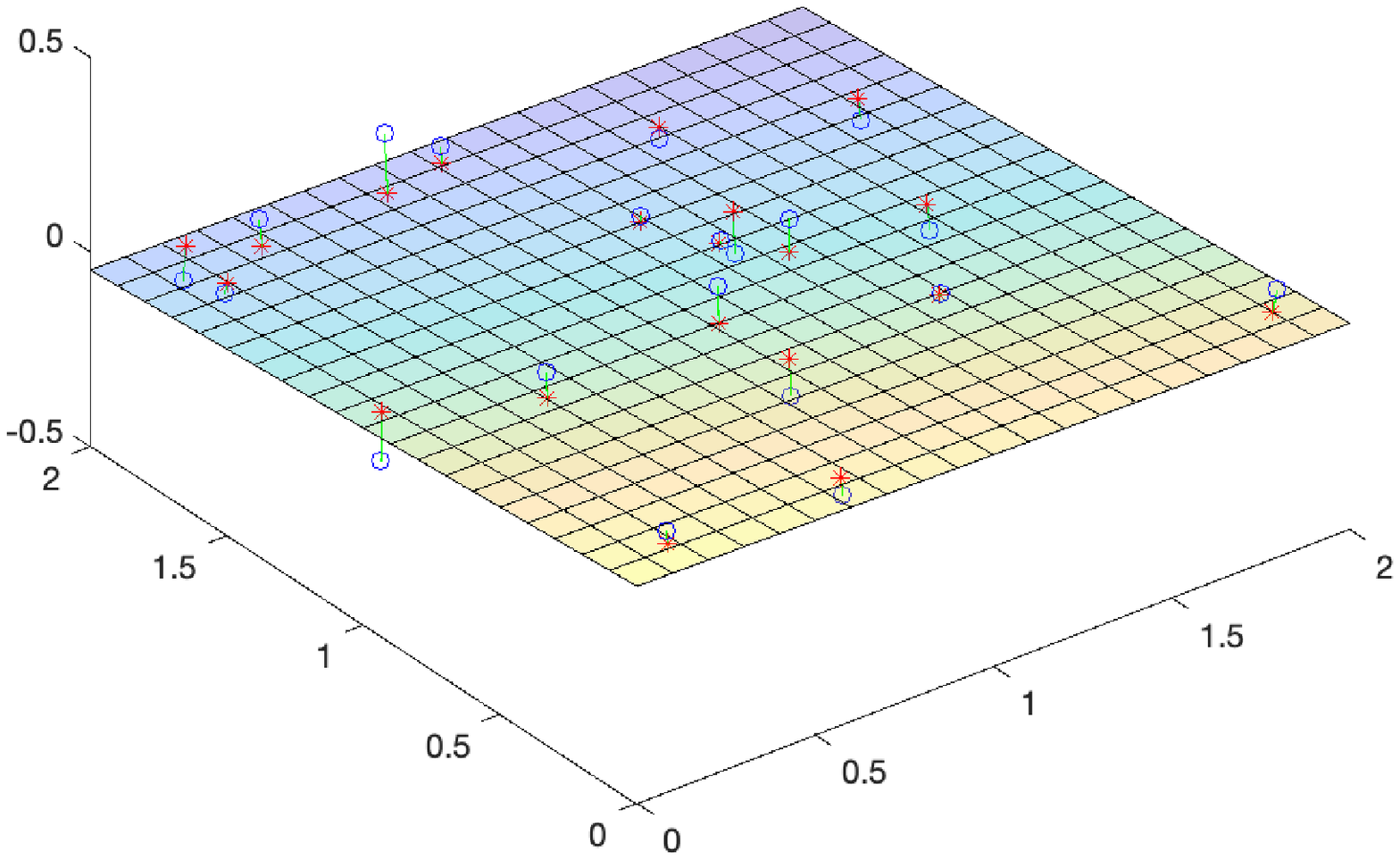}\quad
\includegraphics[width=7cm]{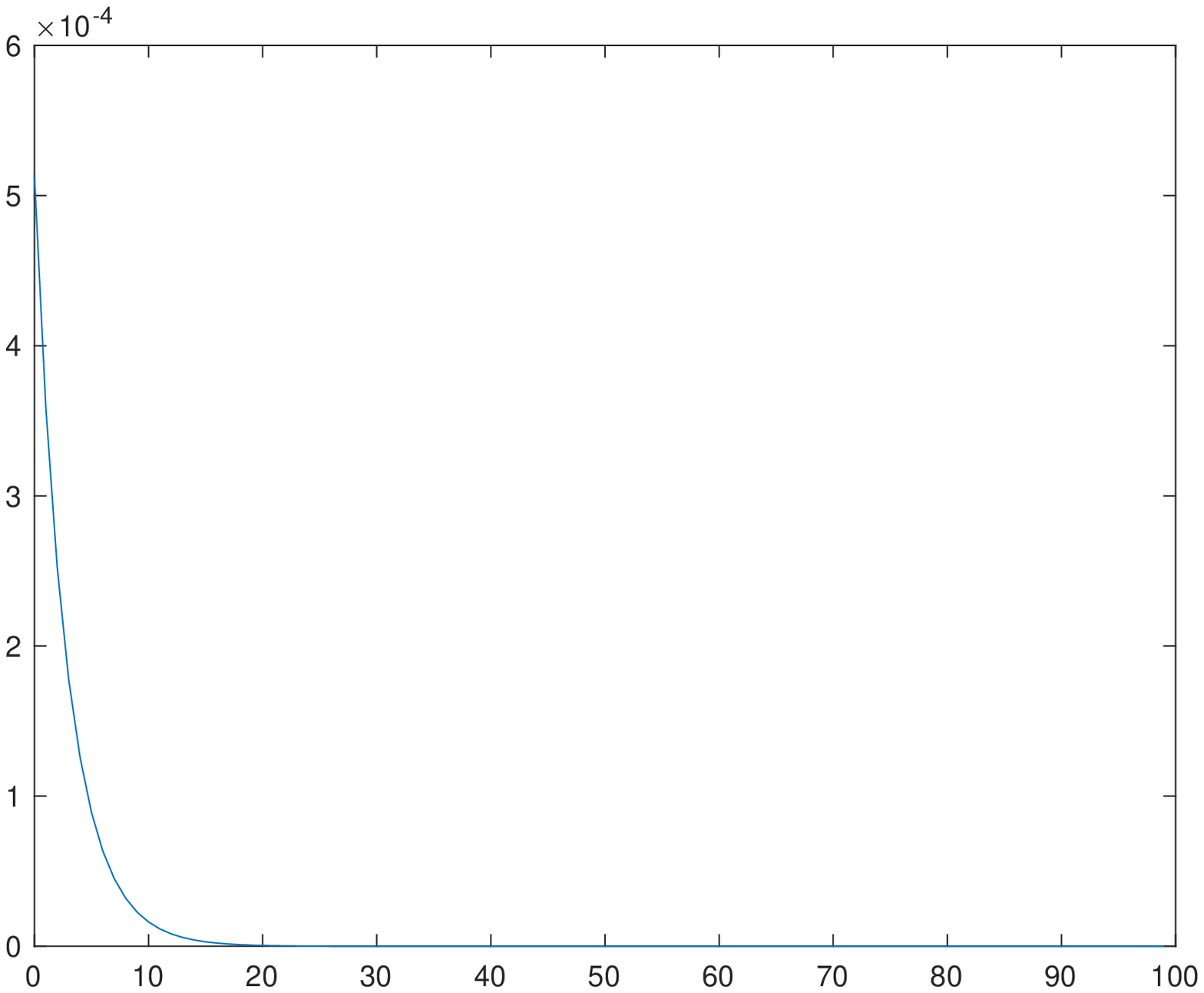}
\caption{If particles $\{x_i\}_{i=1}^N$ follows system \eqref{C-1} with $n=3$, then the particles are aligned on the same plane(Left figure). The right graph plots the time evolution of the average volume of tetrahedrons $x_ix_jx_kx_\ell$. In this simulation, we choose the initial data as the perturbed points from the plane.}\label{Fig2}
\end{figure}

In Figure \ref{Fig2}, we move on to the case of $(d,n)=(3,3)$ and choose initial data as randomly perturbed from  a given plane. The left figure demonstrates that all particles eventually approach to the plane. For the right figure, we plot a temporal evolution of the averaged volume of tetrahedrons with vertices $x_i,x_j,x_k,x_\ell$. We were able to numerically verify that the volume tends to zero as time goes to infinity, as our analytical proof guaranteed. \newline

However, the computational cost of solving system \eqref{C-1} is of order
\begin{equation} \label{C-30}
O(N^ndI),
\end{equation} 
where $N$ is the number of particles in $\bbr^d$, $d$ is the dimension of the space, and $I$ is the number of iterations. For $n\geq2$, the total computational cost  increased rapidly for larger $N\gg n$. Thus, we will provide the reduced model which exhibits similar asymptotic behavior to reduce the computational cost.

\section{Reduced model}\label{sec:5}
\setcounter{equation}{0}

In this section, we provide an algorithm which reduces its computational cost by considering a sparse set of given $n$-simplices. Recall functional \eqref{C-0-1}: 
\[
\mathcal V_n(\mathcal X)=\frac{\kappa_n}{2(n+1)N^{n}}\sum_{j_1, \cdots, j_{n+1}=1}^N \mathrm{Vol}_n(x_{j_1}, \cdots, x_{j_{n+1}})^2.
\]
In this sum, we considered all volumes of $n$-simplices made of $n+1$ vectors in $\{x_1, \cdots, x_N\}$. This is why the computational cost is proportional to $N^n$. Instead of considering all volumes of $n$-simplices, we now consider fewer $n$-simplices. This idea was introduced in \cite{DHJ, DHK, JLL} as the random batch method. In these work, the network topology is changed along the time evolution, however, in this section will only consider that the network topology is constant along the time evolution.
We denote the reduced potential of \eqref{C-0-1} by
\begin{align}\label{E-1}
\mathcal V_n^R(\mathcal X) :=\frac{\kappa_n}{2(n+1)|\mathcal{S}|}\sum_{(j_1, \cdots, j_{n+1})\in \mathcal{S}} \mathrm{Vol}_n(x_{j_1}, \cdots, x_{j_{n+1}})^2,
\end{align}
where $\mathcal{S}\in [N]^{n+1}$ is the set of ordered $(n+1)$-tuple which consists of the simplices used to construct the reduced potential and $|\mathcal{S}|$ is the number of simplices in $\mathcal{S}$. To make the model symmetric, we assume that
\[
(j_1, \cdots,j_{n+1})\in \mathcal{S}\quad\Longleftrightarrow\quad(j_{\sigma(1)}, \cdots, j_{\sigma(n+1)})\in \mathcal{S}
\]
for all permutation $\sigma:\{1, 2, \cdots, n+1\}\to\{1, 2, \cdots, n+1\}$. We also define sets $\{\mathcal{S}_i\}_{i=1}^N$ as follows:
\[
\mathcal{S}_i:=\{(j_1, \cdots, j_n)\in [N]^n: (j_1, \cdots, j_n, i)\in \mathcal{S}\}.
\]
If $\mathcal{S}_i$ is an empty set, then the dynamics of $\{x_1,\cdots,x_N\}$ governed by the potential \eqref{E-1} satisfies  $\dot{x}_i=0$. To prevent this issue, we assume that $\mathcal{S}_i\neq \emptyset$ for all $i\in [N]$. Obviously, if the potential \eqref{C-0-1} is replaced by \eqref{E-1}, then we get the following reduced system:
\begin{align}\label{E-2}
\begin{cases}
\displaystyle\dot{x}_i=-\nabla_{x_i}\mathcal V_n^R(\mathcal X)=-\frac{\kappa_n}{2|\mathcal{S}_i|}\sum_{(j_1, \cdots, j_{n})\in\mathcal{S}_i}^N\nabla_{x_i} \mathrm{Vol}_n(x_{j_1}, \cdots, x_{j_n}, x_i)^2,\quad t>0,\\
x_i(0)=x_i^0\in\bbr^d,\quad\forall~i\in[N].
\end{cases}
\end{align}
Then, the  computational cost of our new model becomes 
\begin{align}\label{E-2-1}
O(|\mathcal{S}|dI),
\end{align}
where $N^n$ in \eqref{C-30} is reduced to $|\mathcal S|$ by choosing few simplices. The only difference between \eqref{E-2} and \eqref{C-1} lies on the number in denominator that is changed from $N^n$ to $|\mathcal S|$.

 We now discuss how much new system \eqref{E-2} reduces the amount computations to run compared to \eqref{C-1}. Before considering $n\geq3$, we start with $n=2$. For $x, y, z\in\bbr^d$,
\[
\mathrm{Vol}_2(x, y, z)=0\quad\Longleftrightarrow\quad \{x, y, z\}\text{ are collinear}.
\]
Hence if we have two collinearities  of three points among four points, i.e.,
\[
\mathrm{Vol}_2(x_1, x_2, x_3)=0\quad\text{ and }\quad\mathrm{Vol}_2(x_2, x_3, x_4)=0,
\]
one can deduce that $\{x_1, x_2, x_3, x_4\}$ is also collinear. In other words, 
\[
\mathrm{Vol}_2(x_1, x_2, x_3)^2+\mathrm{Vol}_2(x_2, x_3, x_4)^2=0\quad\Longrightarrow\quad \{x_1,x_2, x_3, x_4\}\text{ are collinear}.
\]
This indicates that even if we choose two simplices among all possible cases (in fact, $\binom{4}{3}=4$ cases) in the total potential, we are able to obtain a convergence of $\{x_1, x_2, x_3, x_4\}$ to the same line. By using a similar argument to Theorem \ref{T4.1}, the following theorem can be obtained. 

\begin{theorem}\label{T5.1}
Let $\{x_i\}_{i=1}^N$ be a solution to system \eqref{E-2}. Then, there exists $x_i^\infty$ such that
\[
\lim_{t\to\infty}x_i(t)=x_i^\infty,\quad\forall~i\in[N].
\]
Furthermore, if $(j_1, \cdots, j_{n+1})\in \mathcal{S}$, then there exists $P^\infty\subset \mathcal{A}^{n-1}_d$ such that $\{x_{j_1}^\infty, \cdots, x_{j_{n+1}}^\infty\}\subset P^\infty$.
\end{theorem}

If $\mathcal{S}=\{1,\cdots,N\} ^{n+1}$, then for every $(j_1, \cdots, j_{n+1})\in \mathcal{S}$ we can find an $(n-1)$-dimensional affine subspace $P^\infty$ satisfying
$\{x_{j_1}^\infty, \cdots, x_{j_{n+1}}^\infty\}\subset P^\infty$. Since an $(n-1)$-dimensional affine subspace can be determined by choosing $n$ points along it, we know that $\{x_1, \cdots, x_n, y\}\subset P^\infty_1$ and $\{x_1, \cdots, x_n, z\}\subset P^\infty_2$ for some $P_1, P_2\in \mathcal{A}_d^{n-1}$, and this leads to the existence of some $\{x_1, \cdots, x_n, y,z\}\subset P^\infty\in \mathcal{A}_d^{n-1}$. More generally, if $\{x_1, \cdots, x_n, x_\ell\}\subset P^\infty_\ell\in\mathcal{A}_d^{n-1}$ for all $n+1\leq\ell\leq N$ and if there is no $Q\in\mathcal{A}_d^{n-2}$ which contains $\{x_1, \cdots, x_n\}$, then there exists $\mathcal{P}\in\mathcal{A}_d^{n-1}$  containing $\{x_1, \cdots, x_N\}$.\\

We choose a set of base points $\{x_{j_1},\cdots,x_{j_n} \}$ for distinct $j_\ell\in[N]$  and put all simplices including those base points into $\mathcal{S}$, i.e.,
\[
[(j_1, \cdots, j_n, i)]\subset\mathcal{S}
\]
for all $i\in [N]\backslash\{j_1, \cdots, j_n\}$. This action increases the number of element of $\mathcal{S}$ by $n!\times (N-n)=O(Nne^n)$. If $\{x_{j_1}^\infty, \cdots, x_{j_n}^{\infty}\}$ is contained in an $(n-2)$-dimensional affine subspace, then we cannot guarantee that $\{x_1^\infty, \cdots, x_N^\infty\}$ is contained in the same $(n-1)$-dimensional affine subspace. To prevent this issue, we have to choose more sets of base points, say, $\{(j_1^k, \cdots, j_n^k)\}$ for $1\leq k\leq B$. Here, $B$ is the number of base points independent to the number of particles $N$. If we choose $B$ sets of base points, then we have
\[
|\mathcal{S}|\leq B\cdot n!\cdot (N-n)=O(BNne^n).
\]
Here, we used inequality instead of equality since some pairs can be counted more than once. Then the computational cost introduced in \eqref{E-2-1} is
\[
O(BNne^ndI).
\]
Hence, we reduced the order of $N$ in the computational cost.

\begin{corollary}\label{C5.1}
Let $\{x_i\}_{i=1}^N$ be a solution to system \eqref{E-1} for some $n\geq1$. For a given index set
\[
(j_1, j_2, \cdots, j_n)\in [N]^n
\]
satisfying  $1\leq j_1<j_2<\cdots <j_n\leq N$, suppose that  $\mathcal{S}$ contains $(j_1, j_2, \cdots, j_n, i)$ for all $i\in [N]\backslash\{j_1, \cdots, j_n\}$. Then, there exists an $(n-1)$-dimensional affine subspace $P^\infty\subset \mathbb{R}^{d}$ such that 
\[\lim_{t\to\infty } x_i(t)=:x_i^\infty\in P^\infty,\quad \forall~i\in [N] \]
for generic initial data $\{x_i^0\}_{i=1}^N$.
\end{corollary}

\begin{proof}
It follows from Theorem \ref{T5.1} that  for any index $i\in [N]\backslash\{j_1, \cdots, j_n\}$, the set 
\[
\{x_{j_1}^\infty, \cdots, x_{j_n}^\infty, x_i^\infty\}
\]
is contained in an $(n-1)$-dimensional affine subspace $P_i^\infty\subset \bbr^d$. Hence, we can determine an $(n-1)$-dimensional affine subspace $P^\infty$ which contains $\{x_{j_1}^\infty, \cdots, x_{j_n}^\infty\}$ uniquely, and this implies that
\[
P^\infty=P_i^\infty
\]
for all $i\in [N]\backslash\{j_1, \cdots, j_n\}$. Finally, we conclude $\{x_1^\infty, \cdots, x_N^\infty\}\subset P^\infty$.
\end{proof}

\begin{example}\label{E5.1}
As an example of $\mathcal{S}$ in Corollary \ref{C5.1}, we provide the following two sets denoted by $\mathcal{S}_2$ and $\mathcal{S}_3$ for $n=2$ and $n=3$, respectively.

\noindent(1) When $n=2$,  $\mathcal{S}_2$ is defined as 
\[
[(i, j, k)]\subset\mathcal{S}_2\quad\text{if and only if}\quad (i, j, k)=(1, 2,\ell),\quad\forall~3\leq \ell\leq N.
\]

\noindent(2) Similarly for $n=3$, $\mathcal S_3$ is defined as 
\[
[(i, j, k, m)]\subset\mathcal{S}_3\quad\text{if and only if}\quad (i, j, k)=(1, 2,3,\ell),\quad\forall~4\leq \ell\leq N.
\]

\end{example}

Now, we provide two numeric results for $\mathcal{S}_2$ and $\mathcal{S}_3$ introduced in Example \ref{E5.1}. For numerical implementation, same parameters are chosen as in the previous section.  We clearly see that Figures \ref{Fig3} and \ref{Fig4} show the exactly same asymptotic behaviors with Figures \ref{Fig1} and \ref{Fig2}, respectively.

\begin{figure}[h]
\includegraphics[width=7cm]{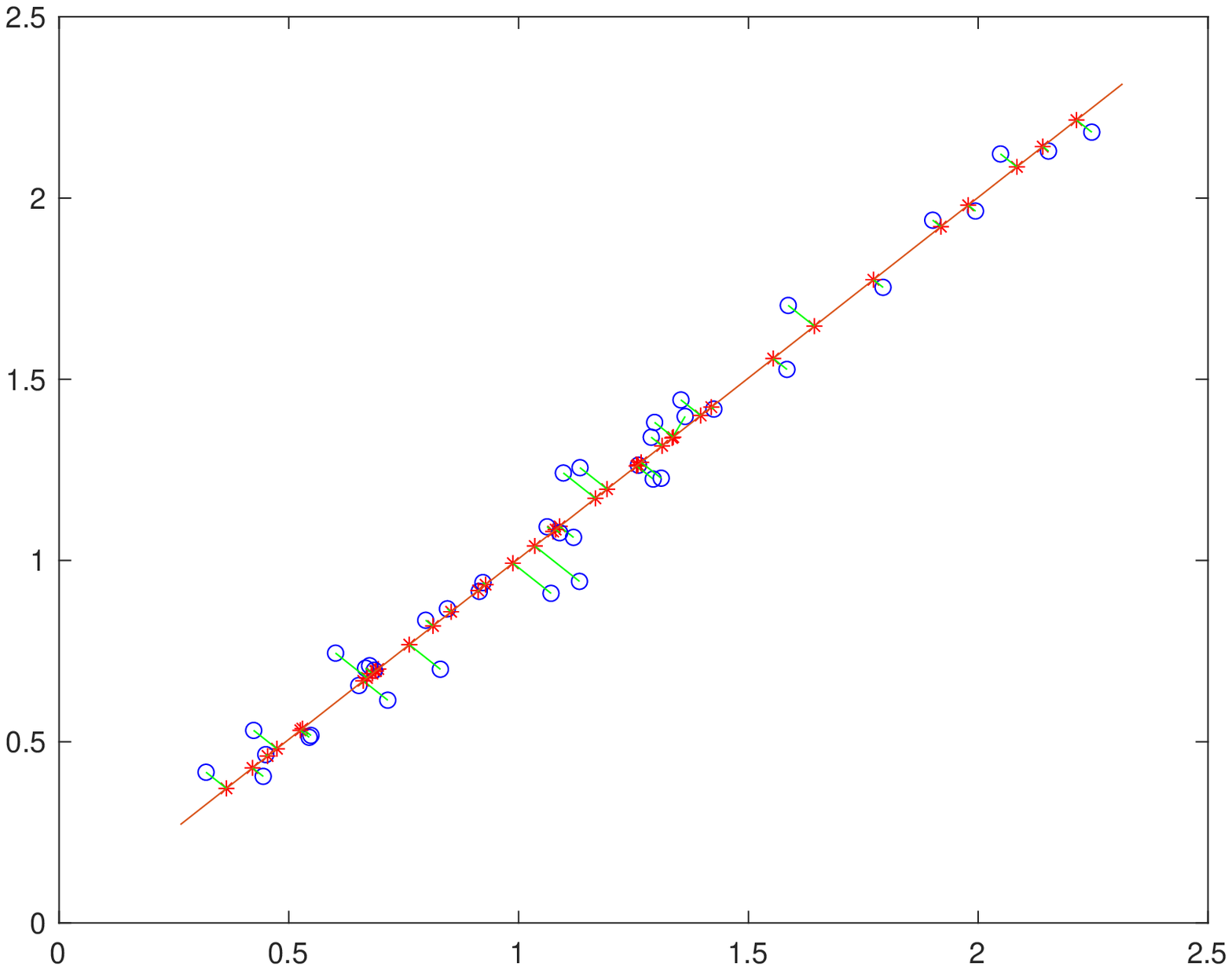}\quad
\includegraphics[width=7cm]{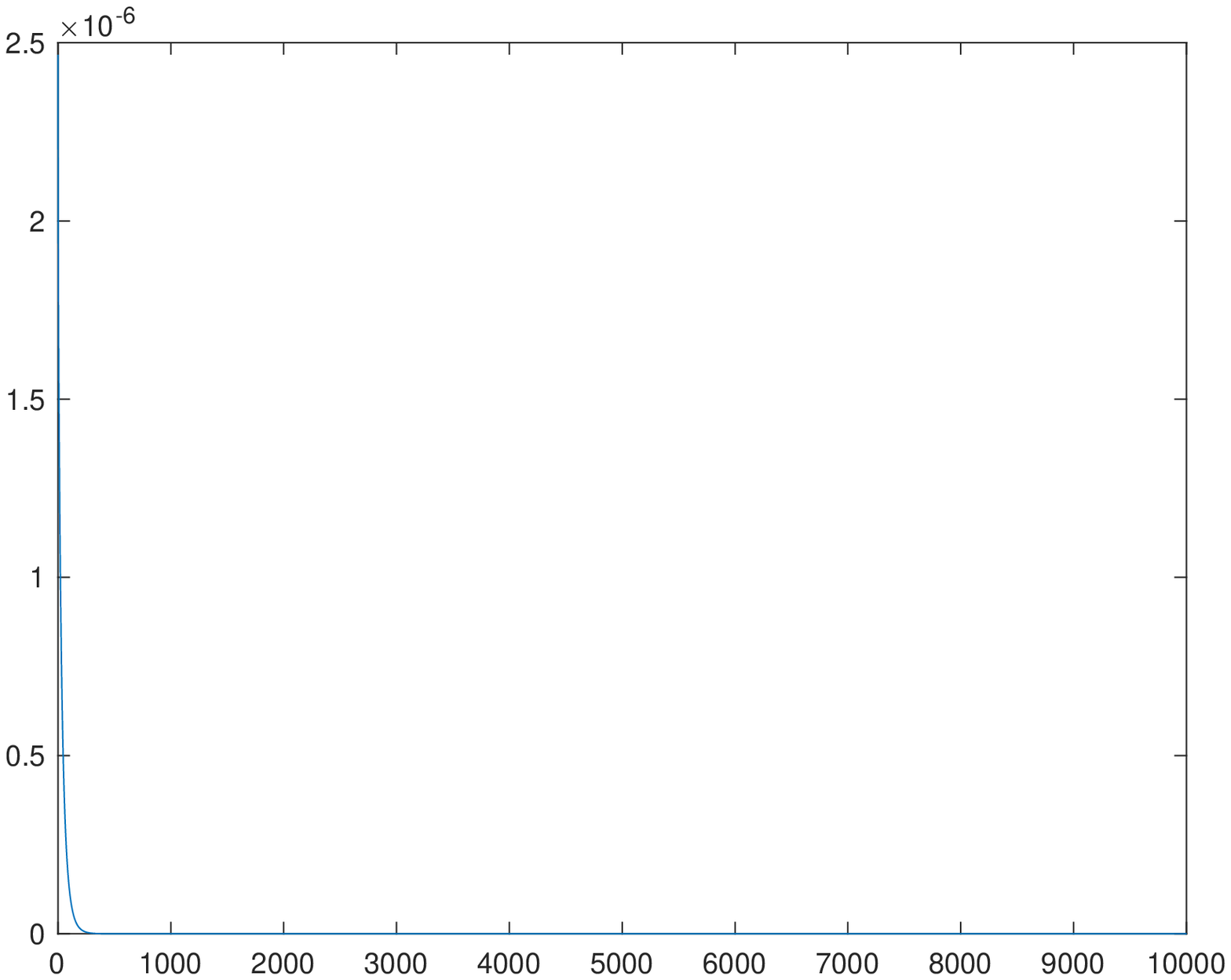}
\caption{If particles $\{x_i\}_{i=1}^N$ follows system \eqref{C-1} with $n=2$ and $\mathcal{S}=\mathcal{S}_2$ introduced in Example \eqref{E5.1}, then the particles are aligned on the same line although the interactions are reduced(Left figure). The right graph plots the temporal evolution of the average area of triangles $x_ix_jx_k$. In this simulation, we choose the initial data as the perturbed points from the line. }\label{Fig3}
\end{figure}

\begin{figure}[h]
\includegraphics[width=7cm]{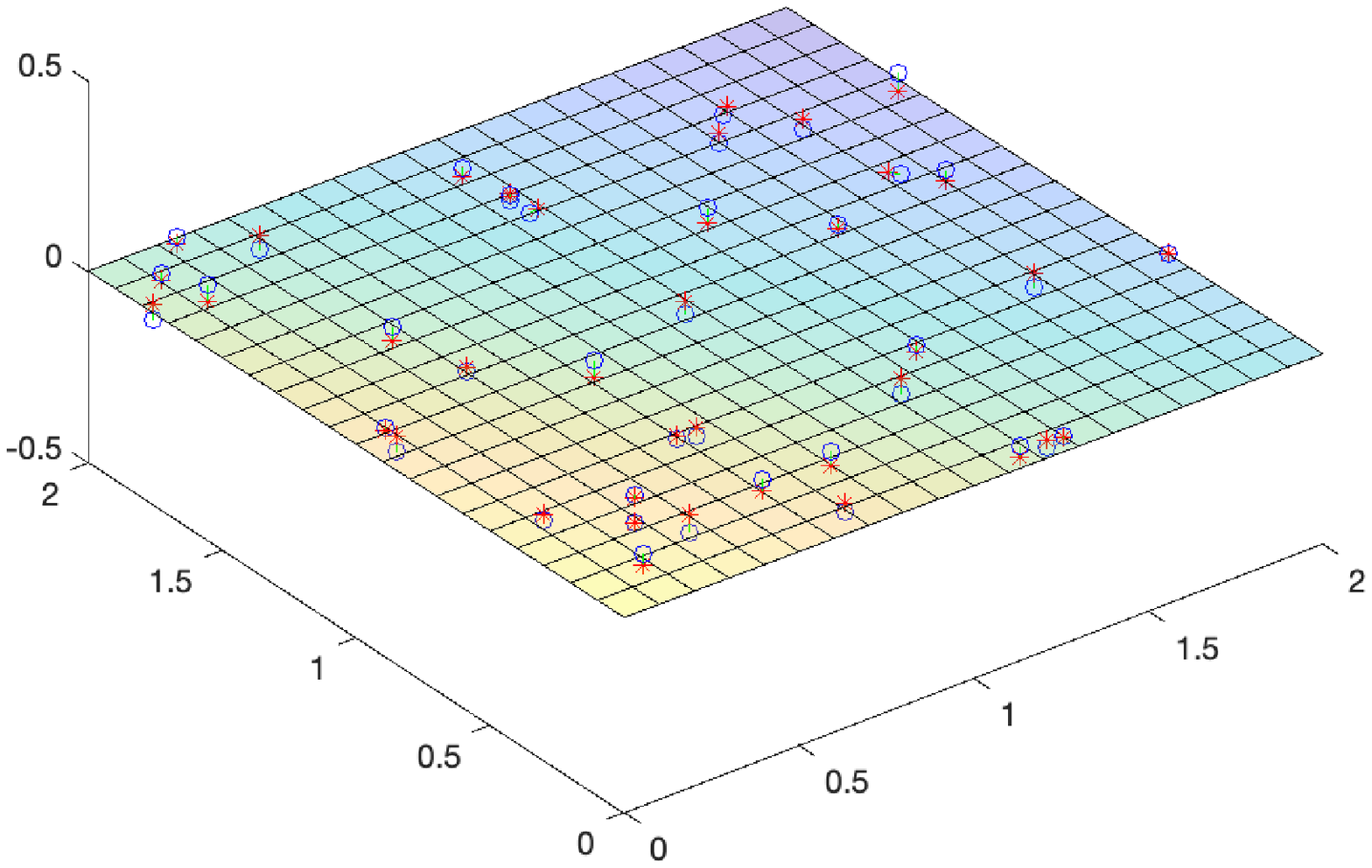}\quad
\includegraphics[width=7cm]{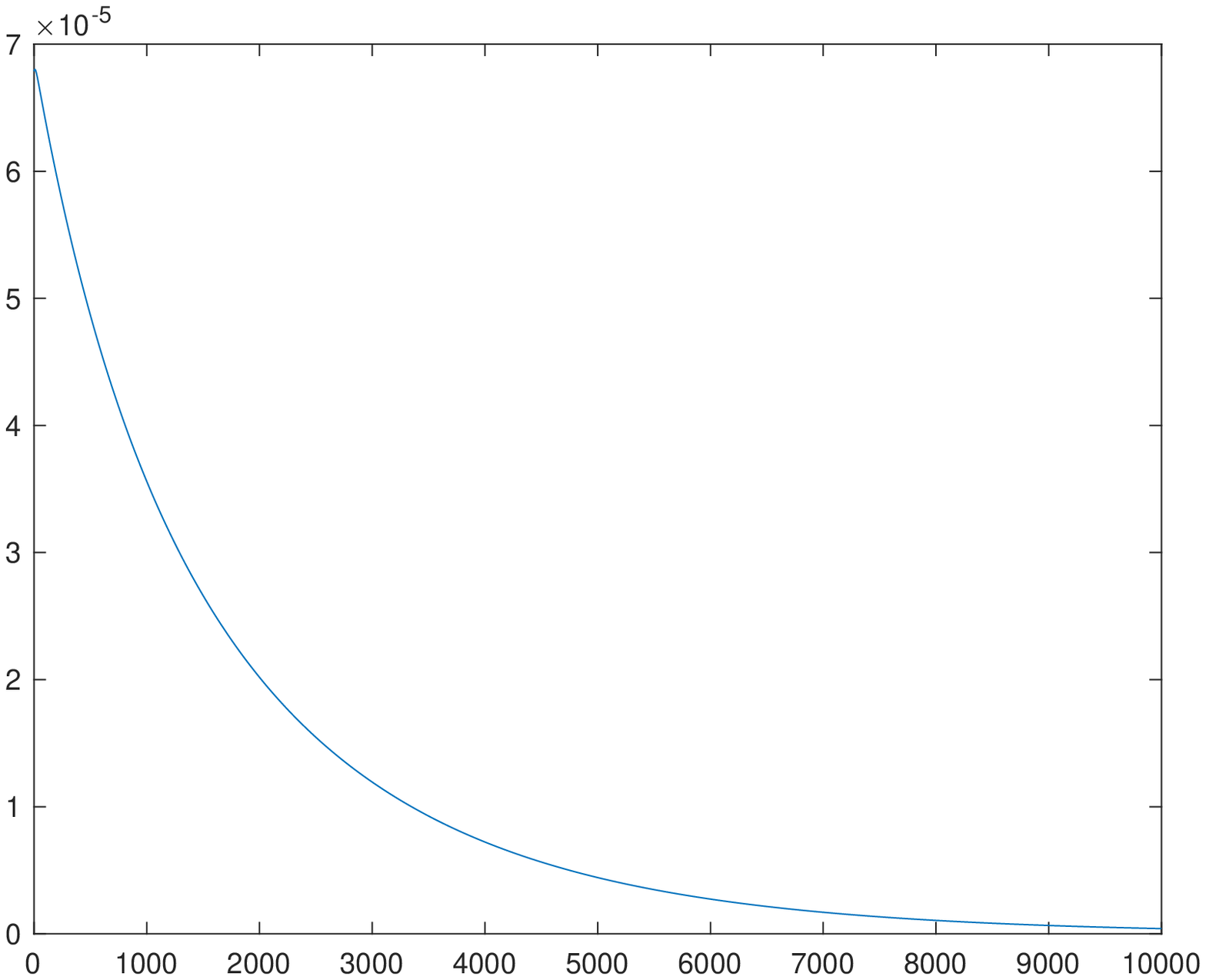}
\caption{If particles $\{x_i\}_{i=1}^N$ follows system \eqref{C-1} with $n=3$ and $\mathcal{S}=\mathcal{S}_3$ introduced in Example \eqref{E5.1}, then the particles are aligned on the same plane although the interactions are reduced(Left figure). The right graph plots the temporal evolution of the average volume of tetrahedrons $x_ix_jx_kx_\ell$. In this simulation, we choose the initial data as the perturbed points from the plane.}\label{Fig4}
\end{figure}

\section{Conclusion} \label{sec:6}
\setcounter{equation}{0}

We studied a generalized linear consensus model in the way that a simplex structure becomes higher. Note that the existing linear consensus model is written as a gradient flow with the sum of squared distances between each two points, and the asymptotic behavior of the model is the convergence of all points to a single point. In this work, we interpreted the model under a higher simplicial framework. To be more specific, since distance between points would be regarded as $1$-dimensional volume between 0-simplices, they can be generalized as $n$-dimensional volume and $(n-1)$-simplices. To this end, we suggested a new gradient flow on higher simplicies and show that a solution gathers to a common $(n-1)$-dimensional affine subspace. This result exactly coincides with the case of $n=1$. Finally, we provide several numerical examples to visualize the theoretical results and reduced model with fewer interaction terms which exhibits similar asymptotic properties.

\section*{Conflict of Interest}
\textbf{The authors have no conflicts of interest to disclose.}

\section*{Data Availability}
\textbf{The data that support the findings of this study are available within the article.}

\appendix
\section{Proof of Lemma \ref{L4.1}}\label{app:A}

For $n=1$, since \eqref{C-1} reduces to the linear consensus model, desired assertion directly follows. On the other hand for $n=2$, we use Example \ref{E2.1}(1) to see that \eqref{C-1} becomes   
\begin{align*}
\dot{x}_i&=-\frac{\kappa_2}{2N^2}\sum_{k, \ell=1}^N \nabla_{x_i} \mathrm{Vol}_2(x_i, x_k, x_\ell)^2\\
&=-\frac{\kappa_2}{2N^2}\sum_{k,\ell=1}^N \nabla_{x_i}\left(\frac{1}{16}(2d_{ik}^2d_{kv}^2+2d_{i\ell}^2d_{k\ell}^2+2d_{ik}^2d_{i\ell}^2-d_{ik}^4-d_{k\ell}^4-d_{i\ell}^4)\right)\\
&=-\frac{\kappa_2}{8N^2}\sum_{k, \ell=1}^N\left((d_{k\ell}^2+d_{i\ell}^2-d_{ik}^2)(x_i-x_k)+(d_{k\ell}^2+d_{ik}^2-d_{i\ell}^2)(x_i-x_\ell)\right)\\
&=\frac{\kappa_2}{8N^2}\sum_{k,\ell=1}^N\left(d_{k\ell}^2(x_k+x_\ell-2x_i)+(d_{i\ell}^2-d_{ik}^2)(x_k-x_\ell)\right).
\end{align*}
Thus, we find
 \begin{equation} \label{C-10}
\dot{x}_i-\dot{x}_j=\frac{\kappa_2}{8N^2}\sum_{k, \ell=1}^N\left(-2d_{k\ell}^2(x_i-x_j)+(d_{i\ell}^2-d_{ik}^2-d_{j\ell}^2+d_{jk}^2)(x_k-x_\ell)\right).
\end{equation} 
We observe from the definition $d_{ij}= \|x_i - x_j\|$:
\[
d_{i\ell}^2-d_{ik}^2-d_{j\ell}^2+d_{jk}^2=2\langle x_i-x_j, x_k-x_\ell\rangle.
\]
Then, \eqref{C-10} becomes 
\[
\dot{x}_i-\dot{x}_j=\frac{\kappa_2}{4N^2}\sum_{k, \ell=1}^N\left(-d_{k\ell}^2(x_i-x_j)+\langle x_i-x_j, x_k-x_\ell\rangle (x_k-x_\ell)\right),
\]
which yields 
\begin{align*}
\frac{\dd}{\dt}\|x_i-x_j\|^2 = \frac{\kappa_2}{2N^2}\sum_{k, \ell=1}^N\left(-\|x_k - x_\ell\|^2\|x_i - x_j\|^2 +\langle x_i-x_j, x_k-x_\ell\rangle^2\right) \leq0 .
\end{align*}
We consider the case of $n\geq3$. For $x_i, x_{j_1}, \cdots, x_{j_n} \in \bbr^d$, define the following two affine sets: 
\begin{align*}
&P_{j_1j_2\cdots j_n}:=\left\{x:~x=\sum_{\alpha=1}^n a_\alpha x_{j_\alpha}\quad\text{where}\quad \sum_{\alpha=1}^n a_\alpha=1
\right\}, \\
& Q_{j_1j_2\cdots j_n}:=\left\{x:~x=\sum_{\alpha=1}^n a_\alpha x_{j_\alpha}\quad\text{where}\quad \sum_{\alpha=1}^n a_\alpha=0
\right\}.
\end{align*}
We know that  $P_{j_1j_2\cdots j_n}$ is the minimal affine subspace of $\mathbb{R}^d$ containing  $x_{j_1}, \cdots, x_{j_n}$, 
\[
P_{j_1j_2\cdots j_n}-P_{j_1j_2\cdots j_n}=Q_{j_1j_2\cdots j_n}.
\] 
We denote  $H^i_{j_1j_2\cdots j_n}$ the orthogonal projection of $x_i$ onto $P_{j_1j_2\cdots j_n}$, i.e.,
\[
H^i_{j_1j_2\cdots j_n}:=\mathrm{argmin}_{y\in P_{j_1j_2\cdots j_n}}\|x_i-y\|.
\]
Then, one can easily verify the following relation:
\[
\mathrm{Vol}_n(x_i, x_{j_1}, \cdots, x_{j_n})=\frac{1}{n}\mathrm{Vol}_{n-1}(x_{j_1}, \cdots, x_{j_n})\times \|H^i_{j_1j_2\cdots j_n}-x_i\|.
\]
For any smooth curve $x_i(t):\mathbb{R}\to \mathbb{R}^{d}$,  we observe 
\[\begin{aligned}
\frac{\dd}{\dt}\|H^i_{j_1j_2\cdots j_n} -x_i \|^2&=2\langle H^i_{j_1j_2\cdots j_n}(t)-x_i(t), \dot{H}^i_{j_1j_2\cdots j_n}(t)-\dot{x}_i(t)\rangle.
\end{aligned} \]
Since  $\dot{H}^i_{j_1j_2\cdots j_n}(t) \in Q_{j_1j_2\cdots j_n}$ and the following relation holds
 \begin{equation}\label{4.9}
 \left\langle y,x_i(t)-H^i_{j_1j_2\cdots j_n}(t)\right\rangle=0,\quad \forall~y\in Q_{j_1j_2\cdots j_n},
 \end{equation}
we calculate the gradient of $\mbox{Vol}_n$ 
\begin{equation}\label{4.10}
\begin{aligned}
\nabla_{x_i}\mathrm{Vol}_n(x_i, x_{j_1}, \cdots, x_{j_n})^2&=\frac{1}{n^2}\mathrm{Vol}_{n-1}(x_{j_1}, \cdots, x_{j_n})^2 \nabla_{x_i}\|H^i_{j_1j_2\cdots j_n}-x_i\|^2\\
&=\frac{2}{n^2}\mathrm{Vol}_{n-1}(x_{j_1}, \cdots, x_{j_n})^2 (x_i-H^i_{j_1j_2\cdots j_n}).
\end{aligned}
\end{equation}

\noindent Now, we substitute the relation \eqref{4.10} into system \eqref{C-1} to obtain
\begin{align} \label{C-11}
\begin{aligned}
\dot{x}_i-\dot{x}_k&=-\frac{\kappa_n}{2N^n}\sum_{j_1, \cdots, j_n=1}^N \frac{2}{n^2}\mathrm{Vol}_{n-1}(x_{j_1}, \cdots, x_{j_n})^2 (x_i-H^i_{j_1j_2\cdots j_n}-x_k+H^k_{j_1j_2\cdots j_n})\\
&=-\frac{\kappa_n}{N^n n^2}\sum_{j_1, \cdots, j_n=1}^N\mathrm{Vol}_{n-1}(x_{j_1}, \cdots, x_{j_n})^2\big((x_i-x_k)-(H^i_{j_1j_2\cdots j_n}-H^k_{j_1j_2\cdots j_n})\big).
\end{aligned}
\end{align}
Since $H^i_{j_1j_2\cdots j_n}-H^k_{j_1j_2\cdots j_n}$ is an element of $Q_{j_1j_2\cdots j_n}$, one can use \eqref{4.9} to obtain

\[\left\langle H^i_{j_1j_2\cdots j_n}-H^k_{j_1j_2\cdots j_n},x_i-H^i_{j_1j_2\cdots j_n}\right\rangle=\left\langle H^i_{j_1j_2\cdots j_n}-H^k_{j_1j_2\cdots j_n},x_k-H^k_{j_1j_2\cdots j_n}\right\rangle=0, \]
which gives 
\begin{equation} \label{C-12}
\left\langle x_i-x_k,(x_i-x_k)-(H^i_{j_1j_2\cdots j_n}-H^k_{j_1j_2\cdots j_n})\right\rangle=\left\|(x_i-x_k)-(H^i_{j_1j_2\cdots j_n}-H^k_{j_1j_2\cdots j_n})\right\|^2.
 \end{equation}
Therefore, we multiply \eqref{C-11} with $x_i -x_k$ and use \eqref{C-12} to establish the desired estimate:
\begin{equation*}
\begin{aligned}
&\frac{\dd}{\dt}\|x_i-x_k\|^2\\
&=-\frac{2\kappa_n}{N^nn^2}\sum_{j_1,\cdots, j_n=1}^N \mathrm{Vol}_{n-1}(x_{j_1}, \cdots, x_{j_n})^2\left\|(x_i-H^i_{j_1j_2\cdots j_n})-(x_k-H^k_{j_1j_2\cdots j_n})\right\|^2\leq 0.
\end{aligned}
\end{equation*}

\section{Proof of Lemma \ref{L4.2}}\label{app:B}

Since each $H^i_{j_1j_2\cdots j_n}$ is an element of $P_{j_1j_2\cdots j_n}$, we get 
	\[H^i_{j_1j_2\cdots j_n}-\frac{1}{N}\sum_{k=1}^N H^k_{j_1j_2\cdots j_n}\in Q_{j_1j_2\cdots j_n},\quad \forall~i\in [N]. \]
	Then, one can use \eqref{4.9} to obtain 
	\[\left\langle x_\ell-H_{j_1j_2\cdots j_n}^\ell,H^i_{j_1j_2\cdots j_n}-\frac{1}{N}\sum_{k=1}^N H^k_{j_1j_2\cdots j_n}\right\rangle=0,\quad \forall~i,\ell\in [N], \]
	which yields  
\[\begin{aligned}
&\left\langle x_i-x_c,(x_i-H^i_{j_1j_2\cdots j_n})-\frac{1}{N}\sum_{k=1}^N(x_k-H^k_{j_1j_2\cdots j_n})\right\rangle\\
&=\left\langle (x_i-x_c)-\left(H^i_{j_1j_2\cdots j_n}-\frac{1}{N}\sum_{k=1}^N H^k_{j_1j_2\cdots j_n}\right),(x_i-H^i_{j_1j_2\cdots j_n})-\frac{1}{N}\sum_{k=1}^N(x_k-H^k_{j_1j_2\cdots j_n})\right\rangle\\
&=\left\|(x_i-H^i_{j_1j_2\cdots j_n})-\frac{1}{N}\sum_{k=1}^N(x_k-H^k_{j_1j_2\cdots j_n}) \right\|^2\geq 0.
\end{aligned} \]	
Now, we substitute the relation \eqref{4.10} into system \eqref{C-1} to obtain	
\[\begin{aligned}
&\frac{\dd}{\dt}\|x_i - \bar x\|^2\\
&=-\frac{2\kappa_n}{N^nn^2}\sum_{j_1,\cdots, j_n=1}^N \mathrm{Vol}_{n-1}(x_{j_1}, \cdots, x_{j_n})^2\left\|(x_i-H^i_{j_1j_2\cdots j_n})-\frac{1}{N}\sum_{k=1}^N(x_k-H^k_{j_1j_2\cdots j_n}) \right\|^2\leq 0.
\end{aligned} \]
Finally, since the barycenter $\bar x$ is a constant of motion, we can conclude the desired result.	

\section{Proof of Theorem \ref{T4.1}}\label{app:C}

Let $X^\infty:=\{x_i^\infty\}_{i=1}^N$ be an equilibrium to \eqref{C-1}. Then, it satisfies 
\[
\mathrm{Vol}_{n-1}(x_{j_1}^\infty, \cdots, x_{j_n}^\infty)^2\left\|(x_i^\infty-H^{i,\infty}_{j_1j_2\cdots j_n})-(x_k^\infty-H^{k,\infty}_{j_1j_2\cdots j_n})\right\|^2=0, \quad \forall i,k,j_1,\cdots,j_n.
\]
We split the proof into two cases. \newline

\noindent (i) If $\mathrm{Vol}_{n-1}(x_{j_1}^\infty, \cdots, x_{j_n}^\infty)\neq0$ for some $j_1, \cdots, j_n$, then
\[
x_i^\infty-x_k^\infty\in Q_{j_1j_2\cdots j_n}, \quad  i,k \in [N].
\]
In particular,  $x_i^\infty-x_{j_1}^\infty\in Q_{j_1\cdots j_n}$ and  $x_i^\infty\in P_{j_1\cdots j_n}$ for all $i\in [N]$.  Since this property holds for all indices $i$ and $P_{j_1\cdots j_n}$ is a subspace of $\mathbb{R}^d$ with dimension  $n-1$, we set $P^\infty:=P_{j_1\cdots j_n}$.\\

\noindent (ii) If  $\mathrm{Vol}_{n-1}(x_{j_1}^\infty, \cdots, x_{j_n}^\infty)=0$ for all $j_1, \cdots, j_n$, we have 
\begin{equation}\label{4.11}
\mathrm{Vol}_{n-1}(0,x_{j_2}^\infty-x_{j_1}^\infty, \cdots, x_{j_n}^\infty-x_{j_1}^\infty)=0,\quad \forall~j_1, \cdots, j_n\in [N].
\end{equation}
Hence,  the dimension of $Q:=\textup{span}\{x_{j_k}^\infty-x_{j_1}^\infty:2\leq k\leq n\}$ is less than or equal to  $n-2$. Suppose to the contrary that the dimension of  $Q$ is greater than or equal to $n-1$.  Then one can find $n-1$ linearly independent subset of $Q$ which spans $n-1$ dimensional subspace of $\mathbb{R}^d$, which contradicts \eqref{4.11}. Since the dimension of $Q$ is  $n-2$, it suffices to define $P^\infty:=x_1^\infty+Q$.


\begin{thebibliography}{99}

\bibitem{BGL}  Benson, A. R., Gleich, D. F., and Leskovec, J.: \textit{Higher-order organization of complex networks}. Science {\bf 353} (2016), 163-166.

\bibitem{CCF} Chacon, J., Chen, M., and Fetecau, R. C.: \textit{Safe coverage of moving domains for vehicles with second order dynamics}, IEEE Trans. Automat. Contr. Early Access

\bibitem{CLC} Caponigro, M., Lai, A. C., Chiara, Lai. A. and Piccoli, B.: \textit{A nonlinear model of opinion formation on the sphere}. Discrete Contin Dyn Syst. A {\bf35} (2015), 4241-4268.

\bibitem{DB} D\"{o}rfler, F. and Bullo, F.: \textit{Synchronization in complex networks of phase oscillators: a survey}, Automatica {\bf50} (2013), 1539-1564.

\bibitem{DHJ} Dong, J.-G., Ha, S.-Y., Jung, J., and Kim, D.: \textit{On the Stochastic Flocking of the Cucker--Smale Flock with Randomly Switching Topologies}. SIAM Journal on Control and Optimization, {\bf 58} (2020), 2332-2353.

\bibitem{DHK} Dong, J.-G., Ha, S.-Y., and Kim, D.: \textit{Emergence of mono-cluster flocking in the thermomechanical Cucker–Smale model under switching topologies}. Analysis and Application, {\bf 19} (2021), 305-342.

\bibitem{GB} Grilli, J., Barab\'{a}s, G., Michalska-Smith, M., and Allesina, S.: \textit{Higher-order interactions stabilize dynamics in competitive network models.} Nature {\bf548} (2017), 210-213. 

\bibitem{JC} J\'{a}cimovi\'{c}, V. and Crnki\'{c}, A.: Low-dimensional dynamics in non-Abelian Kuramoto model on the 3-sphere, Chaos {\bf 28} (2018) 083105.

\bibitem{JLL} Jin, S., Li, L. and Liu, J.-G., \textit{Random Batch Methods (RBM) for interacting particle systems.} J. Comput. Phys. {\bf 400} (2020), 108877.

\bibitem{KJ} Kumar, A. and Jalan, S. : \textit{Explosive synchronization in interlayer phase-shifted Kuramoto oscillators on multiplex networks}. Chaos {\bf31} (2021), 041103.

\bibitem{Ku} Kuramoto, Y.: \textit{Self-entrainment of a population of coupled non-linear oscillators}. Int. Symp. on Mathematical Problems in Theoretical Physics (Lecture Notes in Physics vol 39) ed H Araki (Berlin: Springer) (1975), 420-422.

\bibitem{Ku2} Kuramoto, Y.: \textit{Chemical Oscillations, Waves and Turbulence}, Springer-Verlag, Berlin (1984).

\bibitem{Lo1} Lohe, M. A.: \textit{Higher-order synchronization on the sphere.} J. Phys. Complex. {\bf3} (2022), 015003.

\bibitem{Lo2} Lohe, M. A.: \textit{Non-Abelian Kuramoto models and synchronization}. J. Phys. A: Math. Theor. {\bf42} (2009), 395101.

\bibitem{Lo3} Lohe, M.A.: \textit{Combined higher-order interactions of mixed symmetry on the sphere}. Chaos {\bf 32} (2022), 023114.

\bibitem{MPG}  Markdahl, J., Proverbio, D., and Goncalves, J.: \textit{Robust synchronization of heterogeneous robot swarms on the sphere}. In 2020 59th IEEE Conference on Decision and Control (CDC), pages 5798–5803. IEEE, 2020.

\bibitem{SB} Singh, P., Baruah, G.: \textit{Higher order interactions and species coexistence.} Theor Ecol {\bf14} (2021), 71-83.

\bibitem{So} Sommerville, D. M. Y.: \textit{An Introduction to the Geometry of n Dimensions.} New York: Dover Publications (1958).

\bibitem{TB} Topaz, C. M. and Bertozzi, A. L.: \textit{Swarming patterns in a two-dimensional kinematic model for biological groups}. SIAM J. Appl. Math., {\bf 65} (2004), 152–174.

\bibitem{YMB} Yuanzhi Li, Margaret M Mayfield, Bin Wang, Junli Xiao, Kamil Kral, David Janik, Jan Holik, Chengjin Chu: \textit{Beyond direct neighbourhood effects: higher-order interactions improve modelling and predicting tree survival and growth.} National Science Review {\bf 8} (2021) , nwaa244.

\bibitem{Wi} Winfree, A.: \textit{Biological rhythms and the behavior of populations of coupled oscillators}. J. Theoret. Biol. {\bf 16} (1967), 15-42.

























%
%
%
%
%


\end{thebibliography}
\end{document}